\documentclass[12pt, a4paper]{amsart}
\usepackage{amsmath}
\usepackage{stmaryrd}
\usepackage{amssymb}
\usepackage{amsfonts}
\usepackage{amsthm}
\usepackage{xcolor}
\usepackage{enumerate}
\PassOptionsToPackage{hyphens}{url}\usepackage{hyperref}

\newtheorem{theo}{Theorem}[section]
\newtheorem*{theo*}{Theorem}

\newtheorem{defi}[theo]{Definition}

\newcommand{\N}{\mathbb{N}}

\title{Two identities involving Cohen-Ramanujan expansions}

\begin{document}
\author[A Chandran]{Arya Chandran}
\address{Department of Mathematics, CVV Institute of Science and Technology, Chinmaya Vishwa Vidyapeeth (Deemed to be University),  Anthiyal-Onakkoor Road, Ernakulam, Kerala-686667, India}
\email{aryavinayachandran@gmail.com}
\author[K V Namboothiri]{K Vishnu Namboothiri}
\address{Department of Mathematics, 
Baby John Memorial Government College, Chavara, Sankaramangalam, Kollam – 691583,
Kerala, India}
\address{Department of Collegiate Education, Government of Kerala, India}
\email{kvnamboothiri@gmail.com}

\begin{abstract}
An arithmetical function $f$ is said to admit a  \emph{Cohen-Ramanujan expansion} $f(n) := \sum\limits_{r}\widehat{f}(r)c_r^s(n)$, if the series on the right hand side converges for suitable complex numbers $\widehat{f}(r)$. Here $c_r^s(n)$ denotes the Cohen-Ramanujan sum defined by E. Cohen. We deduce here a Cohen-Ramanujan expansion for the Jordan totient function $J_k(n)$. Further, we give an an asymptotic formula for the sum $\sum\limits_{n \leq N} \frac{J_a(n)}{n^a} \frac{J_b(n+h)}{(n+h)^b}$ using the expansion we derive.
\end{abstract}

 \maketitle
\section{Introduction}
For some well-known arithmetical functions, Srinivasa Ramanujan introduced infinite Fourier series like expansions in the form $\sum\limits_{r}a_rc_r(n)$ in \cite{ramanujan1918certain}. In this expression, $c_r(n)$  is the \emph{Ramanujan sum} defined by
 \begin{align*}
c_r(n):=\sum\limits_{\substack{{m=1}\\(m,r)=1}}^{r}e^{\frac{2 \pi imn}{r}}.
 \end{align*}
 When such an infinite series representation exists for an arithmetical function $f$ (mostly only pointwise convergent), $f$ is said to possess a \emph{Ramanujan-Fourier series expansion} or simply a \emph{Ramanujan expansion}. Such expansions and several conditions for the existence of such expansions were provided by many authors in papers like \cite{hardy1921note},  \cite{lucht2010survey} and \cite{murty2013ramanujan}.

E. Cohen generalized the Ramanujan sum in  \cite{cohen1949extension} defining the \emph{Cohen-Ramanujan sum} by
\begin{align}\label{gen-ram-sum}
c_r^s(n) :=\sum\limits_{\substack{h=1\\{(h,r^s)_s=1}}}^{r^s}e^{\frac{2\pi i n h}{r^s}}.
\end{align}
 In Cohen-Ramanujan sum, the symbol $(m,n)_{s}$ stands for the \emph{generalized GCD} of $m$ and $n$ which is defined to be equal to the largest positive integer $l^s$ such that $l^s|m$ and $l^s|n$ with $l\in\N$.  When $s=1$,   this sum reduces to the usual Ramanujan sum.

  Analogous to the Ramanujan expansions, an arithmetical function $f$ is said to admit a \emph{Cohen-Ramanujan expansion} in the form
 \begin{align*}
 f(n) := \sum\limits_{r}\widehat{f}(r)c_r^s(n^s),
 \end{align*}
if the series on the right hand side converges for suitable complex numbers $\widehat{f}(r)$. These authors derived some conditions for the existence of such expansions   in \cite{chandran2023ramanujan}. After this, some asymptotic formulae involving Cohen-Ramanujan expansions were also derived by these authors in \cite{chandran2025asymptotic}.

 E. Cohen  himself gave another generalization of the Ramanujan sum in \cite{cohen1959trigonometric} using $k$-vectors. For a positive integer $k$, a \emph{$k$-vector} is  an ordered set $\{x_i\}=\{x_1,\ldots,x_k\}$ of $k$ integers. Two $k$-vectors $\{x_i\}$ and $\{y_i\}$ are said to be congruent (mod $k,r$) if  $x_i\equiv y_i$ (mod $r$), $i=1,\ldots,k$. For $k\geq 1$, Cohen defined
\begin{align}
c^k(n,r) &= \sum\limits_{(\{x_i\},r)=1}e^{\frac{2\pi i n (x_1+x_2+\ldots x_k)}{r}},\label{cohen-sum2}
\end{align} where $\{x_i\}$ ranges over a residue system (mod $k,r$). For positive integers $k$ and $n$, the \emph{Jordan totient function} $J_k(n)$  \cite[{Section V.3}]{sivaramakrishnan1988classical} is defined  to be the number of ordered sets of $k$ elements from a complete residue system (mod $n$) such that the greatest common divisor of each set is prime to $n$.
A product formula for $J_k$ ( \cite[{Section V.3}]{sivaramakrishnan1988classical}) similar to the one for the Euler totient function $\varphi$ is
  \begin{align}
  J_k(n)=n^k \prod\limits_{\substack{p\mid n\\p\text{ prime}}}\left(1-\frac{1}{p^k}\right)\label{eq:j_k_prod_formula}.
  \end{align}
From this formula, it follows that $J_k(n)$ is multiplicative in $n$ and $ J_1(n)=\varphi(n)$.
 Two positive integers $a,b$ are said to be relatively $k-$prime if $(a,b)_k = 1$. E. Cohen introduced the \emph{Cohen totient function} $\varphi_k(r)$  \cite[{Section V.5}]{sivaramakrishnan1988classical} to denote the number of integers $a$ with $1\leq a\leq r^k$ that are relatively $k-$prime to $r^k$. Hence  $\varphi_1(n) =\varphi(n)$.  Further, $\varphi_k(r) = J_k(r)$ though they are defined in different ways (see the remark after  \cite[{Corollary 5.3}]{sivaramakrishnan1988classical}).

S. Ramanujan  \cite{ramanujan1918certain} derived the infinite series expansion

\begin{align}
\frac{J_k(n)}{n^k}\zeta(k+1) = \sum\limits_{\substack{r=1}}^{\infty}\frac{\mu(r)}{J_{s+1}(r)} c_r(n),
\end{align}
where $k$ is a positive integer, $\mu$ is the M{\"o}bius function and $\zeta$ is the Riemann zeta function. Similar to this, R. Sivaramakrishnan suggested  the identity
\begin{align}\label{Jordan_sum}
\frac{J_k(n)}{n^k}\zeta(s+k) = \sum\limits_{\substack{r=1}}^{\infty}\frac{\mu(r)}{J_{s+k}(r)} c^s(n,r),
\end{align} where $s$ and $k$  are positive integers
as an exercise problem  \cite[Exercise 7]{sivaramakrishnan1988classical}. Analogous to these two identities, we derive in this paper the following identity involving the Cohen-Ramanujan sum \eqref{gen-ram-sum}.
\begin{theo}\label{Coh-sum-J_k}
For positive integers $s,k$ and $n$, we have
 \begin{align*}
 \frac{J_k(n)}{n^k}\zeta(s+k) = \sum\limits_{q=1}^{\infty} \frac{\mu(q)}{J_{s+k}(q)}c_q^s(n^s).
 \end{align*}
\end{theo}

H. Gadiyar, M. Ram Murthy and R. Padma  \cite{gopalakrishna2014ramanujan} derived an asymptotic formula for the sum
$\sum\limits_{n\leq N} f(n)g(n+h),$
where $f$ and $g$ are two arithmetical functions possessing absolutely convergent Ramanujan expansions.
Inspired by this, for arithmetical functions $f$ and $g$ with absolutely convergent Cohen-Ramanujan expansions, an asymptotic formula for $\sum\limits_{\substack{n\leq N}}f(n)g(n+h)$  was provided by these authors in  \cite[Theorem 1.2]{chandran2025asymptotic}. Note that, $f(x)$ is said to be \emph{asymptotic} to $g(x)$ (written symbolically as $f(x)\sim g(x)$) if $\lim\limits_{x\rightarrow \infty}\frac{f(x)}{g(x)}=1$. We use the asymptotic formula \cite[Theorem 1.2]{chandran2025asymptotic} to  derive the following asymptotic identity involving Jordan totient function and it is similar the  identity \cite[Corollary 2]{murty2015error} given by M. Ram Murty and B. Saha. Recall that an integer $k$ is said to be \emph{$s-$power free} if there does not exist any prime $p$ such that $p^s$ divides $k$.

\begin{theo}\label{Asym-J_k}
If $s>1$, $a,b > 1+\frac{s}{2}$ and $h= m^sk$, where $k$ is an $s-$ power free integer, then
\begin{align*}
\sum\limits_{n \leq N} \frac{J_a(n)}{n^a} \frac{J_b(n+h)}{(n+h)^b} \sim N &\prod\limits_{p \mid m}\Big(\big(1-\frac{1}{p^{s+a}}\big)\big(1-\frac{1}{p^{s+b}}\big)+\frac{p^{s-1}}{p^{a+b+2s}}\Big)\\&\times\prod\limits_{p \nmid m}\Big(\big(1-\frac{1}{p^{s+a}}\big)\big(1-\frac{1}{p^{s+b}}\big)-\frac{1}{p^{a+b+2s}}\Big).
\end{align*}
\end{theo}

Although expansions involving Ramanujan sums have been studied extensively by several authors, the class of Cohen–Ramanujan expansions has received comparatively little attention in the literature. We hope that the Cohen–Ramanujan expansion derived in this paper will stimulate further exploration of similar expansions. Moreover, the asymptotic formula established here can serve as a foundation for future investigations, particularly in examining the accuracy of these estimates for varying values of
$s$, which may open up a rich avenue of related problems.

 \section{Proofs of the Results}
 
\begin{proof}[Proof of Theorem \ref{Coh-sum-J_k}]
Recall that the M{\"o}bius function $\mu$  is defined by the properties
\begin{enumerate}
 \item $\mu(1)=1$
 \item $\mu(p_1^{r_1}\times \ldots \times p_k^{r_k}) =
 \begin{cases}
(-1)^k \text{ if } r_1=\ldots = r_k=1\\
  0 \text{ if any of the } r_i\text{ is greater than 1}
  \end{cases}$
 \end{enumerate} and hence $\mu$ is multiplicative and $\mu(p^r)=0$ for any prime $p$ and integer $r>1$. By \cite[Theorem 1]{cohen1949extension}, $c_r^s(n)$ is multiplicative in $r$.
Hence if $s,k\geq 1$, by the multiplicative properties of these arithmetical functions, we have
\begin{align*}
\sum\limits_{q} \frac{\mu(q)c_q^s(n^s)}{J_{s+k}(q)}&=\prod\limits_{\substack{p\\p \text{ prime}}}\left(1+\frac{\mu(p)c_p^s(n^s)}{J_{s+k}(p)}\right)
\\&= \prod\limits_{\substack{p\mid n \\p \text{ prime}}}\left(1+\frac{\mu(p)c_p^s(n^s)}{J_{s+k}(p)}\right)\prod\limits_{\substack{p\nmid n\\p \text{ prime}}}\left(1+\frac{\mu(p)c_p^s(n^s)}{J_{s+k}(p)}\right).
\end{align*}

By the properties of the Cohen-Ramanujan sum \cite[Equation 1.5]{cohen1950extension}, if $p$ is a prime, we have
\begin{align}
 c^s_{p^r}(n) = \begin{cases}
              p^{sr} - p^{s(r-1)}&\text{ if }p^{sr}|n\\
              - p^{s(r-1)}&\text{ if }p^{s(r-1)}|n,\,p^{sr}\nmid n\\
              0&\text{ if }p^{s(r-1)}\nmid n\\
             \end{cases}\label{eq:cohen-ram-prime-prop}
\end{align}
 and so when $p\nmid n$, we have
 \begin{align*}
\frac{\mu(p)c_p^s(n^s)}{J_{s+k}(p)} = \frac{(-1)(-1)}{p^{s+k}(1-1/p^{s+k})}
 \end{align*} and when $p|n$, we have
 \begin{align*}
\frac{\mu(p)c_p^s(n^s)}{J_{s+k}(p)} = \frac{(-1)(p^s-1)}{p^{s+k}(1-1/p^{s+k})}.
 \end{align*}
Therefore,
 \begin{align*}
\sum\limits_{q} \frac{\mu(q)c_q^s(n^s)}{J_{s+k}(q)}&= \prod\limits_{\substack{p\mid n \\p \text{ prime}}}(1-\frac{p^s-1}{p^{s+k}-1})\prod\limits_{\substack{p\nmid n\\p \text{ prime}}}(1+\frac{1}{p^{s+k}-1}).
\end{align*}

But \begin{align*}
     \prod\limits_{\substack{p\nmid n\\p \text{ prime}}}(1+\frac{1}{p^{s+k}-1})  & = \prod\limits_{\substack{p \text{ prime}}}\frac{p^{s+k}}{p^{s+k}-1} \prod\limits_{\substack{p| n\\p \text{ prime}}} \frac{p^{s+k}-1}{p^{s+k}}\\
     & = \prod\limits_{\substack{p \text{ prime}}}\frac{1}{1-p^{-(s+k)}} \prod\limits_{\substack{p| n\\p \text{ prime}}} \frac{p^{s+k}-1}{p^{s+k}}
    \end{align*}
and
\begin{align*}
 \prod\limits_{\substack{p\mid n \\p \text{ prime}}}(1-\frac{p^s-1}{p^{s+k}-1}) = \prod\limits_{\substack{p\mid n \\p \text{ prime}}}\frac{p^s(p^k-1)}{p^{s+k}-1}.
\end{align*}

So  \begin{align*}
\sum\limits_{q} \frac{\mu(q)c_q^s(n^s)}{J_{s+k}(q)} =  \prod\limits_{\substack{p}}\frac{1}{1-p^{-(s+k)}}\prod\limits_{\substack{p\mid n\\p \text{ prime}}}(1-\frac{1}{p^k}).
\end{align*}
Using the Euler product formula \cite[Theorem 280]{hardy1979introduction}, we have
\begin{align*}
\zeta(z) = \prod\limits_{p}\frac{1}{1-p^{-z}}.
\end{align*}
 Further, $\zeta(z)$ converges for all $z$ with $Re(z) \geq 1$. Hence
 \begin{align*}
  \sum\limits_{q} \frac{\mu(q)c_q^s(n^s)}{J_{s+k}(q)} = \zeta(s+k) \frac{J_k(n)}{n^k}
 \end{align*}
 as we claimed in the statement of the theorem.
\end{proof}
For a positive integer $n$, recall that the number of divisors of $n$ is denoted by $\tau(n)$.
\begin{defi}
 For $k,n\in \N$, by $\tau_{k}(n)$, we mean the number of positive integers $l^k$ dividing $n$  where $l\in \N$.
 \end{defi}
 Note that $\tau(n)\leq n$ and $\tau_s(n^s) = \tau(n)$.
 The following asymptotic result is from \cite{chandran2025asymptotic}.
\begin{theo}\cite[Theorem 1.2]{chandran2025asymptotic}\label{con-sum}
Suppose that $f$ and $g$ are two arithmetical functions with absolutely convergent Cohen-Ramanujan expansions
\begin{center}

$f(n) = \sum\limits_{\substack{r}}\widehat{f}(r)c_r^{s}(n)$ and $g(n) = \sum\limits_{\substack{k}}\widehat{g}(k)c_k^{s}(n)$
\end{center}
respectively. Suppose that $\sum\limits_{\substack{r,k}}\vert \widehat{f}(r)\vert \vert\widehat{g}(k) \vert (r^sk^s)^{\frac{1}{2}} \tau_s(r^s) \tau_s(k^s)< \infty$. Then as $N$ tends to infinity, $\sum\limits_{\substack{n \leq N}} f(n)g(n+h)\sim N \sum\limits_{\substack{r}} \widehat{f}(r)\widehat{g}(r) c_r^s(h)$.

\end{theo}

Now we use this asymptotic identity along with the identity in Theorem \ref{Coh-sum-J_k} to prove  Theorem \ref{Asym-J_k}.
\begin{proof}[Proof of Theorem \ref{Asym-J_k}]
If we let $f(n) = \frac{J_a(n)}{n^a}$ and $g(n) = \frac{J_b(n)}{n^b}$, then by Theorem  \ref{Coh-sum-J_k}, we have
$\widehat{f}(r) = \frac{\mu(r)}{J_{s+a}(r)\zeta(s+a)}$ and $\widehat{g}(r) = \frac{\mu(r)}{J_{s+b}(r)\zeta(s+b)}$. Now
\begin{align*}
 \frac{1}{J_s(r)} = \frac{1}{r^s \prod\limits_{\substack{p\mid n}}(1-p^{-s})} \leq \frac{1}{r^s \prod\limits_{\substack{p}}(1-p^{-s})} = \frac{\zeta(s)}{r^s}.
\end{align*}
Hence
\begin{align*}
\sum\limits_{\substack{r,t}}\vert \widehat{f}(r)\vert \vert\widehat{g}(t)& \vert (r^st^s)^{\frac{1}{2}} \tau_s(r^s) \tau_s(t^s)\\ & \leq \sum\limits_{\substack{r,t}} \frac{1}{J_{s+a}(r)\zeta(s+a)} \frac{1}{J_{s+b}(t)\zeta(s+b)}(r^st^s)^{\frac{1}{2}}r^st^s \\& \leq  \sum\limits_{\substack{r,t}}\frac{\zeta(s+a)}{r^{s+a}\zeta(s+a)} \frac{\zeta(s+b)}{t^{s+b}\zeta(s+b)}(r^st^s)^{\frac{1}{2}}r^st^s
\\&  = \sum\limits_{\substack{r,t}} \frac{1}{r^{a-\frac{s}{2}}t^{b-\frac{s}{2}}}.
\end{align*}
Since $a,b > 1+\frac{s}{2}$, the above sum converges. Therefore, by Theorem \ref{Coh-sum-J_k} and Theorem \ref{con-sum}, we may write
\begin{align*}
\sum\limits_{n\leq N} \frac{J_a(n)}{n^a} \frac{J_b(n+h)}{(n+h)^b} \sim N \sum\limits_{r} \frac{\mu(r)}{J_{s+a}(r)\zeta(s+a)}\frac{\mu(r)}{J_{s+b}(r)\zeta(s+b)}c_r^s(h).
\end{align*}
From the definition of the Cohen-Ramanujan sum, we can see that if $h=m^s k$ where  $k$ is $s$-power free, then $c_r^s(h) = c_r^s(m^s)$. Using the multiplicative properties of $J_s, \mu$ and $c_r^s$ (on $r$), we may rewrite the above as
\begin{align*}
\sum\limits_{n\leq N} \frac{J_a(n)}{n^a} &\frac{J_b(n+h)}{(n+h)^b}
\\&\sim N  \prod\limits_{p}(1-\frac{1}{p^{s+a}})(1-\frac{1}{p^{s+b}} )\Big( 1+\dfrac{c_{p}^s(m^s)}{J_{s+a}(p)J_{s+b}(p)}\Big)
\\&=  N \prod\limits_{p}\Big((1-\frac{1}{p^{s+a}})(1-\frac{1}{p^{s+b}})\big(1+\frac{c_{p}^s(m^s)}{(p^{s+a}-1)(p^{s+b}-1)}\big) \Big)
\\&= N \prod\limits_{p}\Big((1-\frac{1}{p^{s+a}})(1-\frac{1}{p^{s+b}})+\frac{c_{p}^s(m^s)}{p^{a+b+2s}} \Big).
\end{align*}
Using the properties \eqref{eq:cohen-ram-prime-prop} of Cohen-Ramanujan sums, we thus conclude that
\begin{align*}
\sum\limits_{n\leq N} \frac{J_a(n)}{n^a} \frac{J_b(n+h)}{(n+h)^b} \sim N &\prod\limits_{p\mid m}\Big((1-\frac{1}{p^{s+a}})(1-\frac{1}{p^{s+b}})+\frac{p^s-1}{p^{a+b+2s}}\Big)\\
&\times \prod\limits_{p\nmid m}\Big((1-\frac{1}{p^{s+a}})(1-\frac{1}{p^{s+b}})-\frac{1}{p^{a+b+2s}} \Big).
\end{align*}
\end{proof}  

\section*{Acknowledgements}
The first author thanks the University Grants Commission of India for providing financial support for carrying out this research work through their Senior Research Fellowship (SRF) scheme.


\end{document}